\documentclass[12pt]{amsart}
\usepackage{amssymb}
\usepackage{t1enc}
\usepackage[latin2]{inputenc}
\usepackage{verbatim}
\usepackage{color}
\usepackage{amsmath,amsfonts,amssymb,amsthm}
\usepackage[mathcal]{eucal}
\usepackage{enumerate}
\usepackage[centertags]{amsmath}
\usepackage{graphics}

\newtheorem{theorem}{Theorem}
\newtheorem{lemma}{Lemma}
\newtheorem{definition}{Definition}
\newtheorem{corollary}{Corollary}

\newtheorem{remark}{Remark}
\numberwithin{equation}{subsection}

\newcommand{\nub}{{\boldsymbol{\nu}}}

\newcommand{\cx}{{}{\scriptstyle{\mathcal X}}}

\begin{document}
\author{R. Duduchava, E. Shargorodsky, and G. Tephnadze}
\title[ unit normal vector field]{ Extension of the unit normal vector field
from a hypersurface}
\address{R. Duduchava, A.Razmadze Mathematical Institute Academy of Sciences
of Georgia 1, M.Alexidze str. Tbilisi 93, Georgia }
\email{duduch@rmi.acnet.ge}
\address{E. Shargorodsky, Department of Mathematics, King's College London,
Strand, London WC2R 2LS, United Kingdom.}
\email{eugene.shargorodsky.AT.kcl.ac.uk}
\address{G. Tephnadze, Department of Mathematics, Faculty of Exact and
Natural Sciences, Tbilisi State University, Chavchavadze str. 1, Tbilisi
0128, Georgia}
\email{giorgitephnadze@gmail.com}
\thanks{The research was supported by Shota Rustaveli National Science
Foundation grant no. 13/14.}
\date{}
\maketitle

\begin{abstract}
It is important in many applications to be able to extend the (outer) unit normal
vector field from a hypersurface to its neighborhood in such a way that the result is a unit gradient field. The aim of the paper is to provide an elementary proof of the existence and uniqueness of such an extension.
\end{abstract}



\section*{Introduction}

\bigskip It is important in many applications to be able to extend the (outer) unit normal
vector field ${\boldsymbol{\nu }}$ from a hypersurface $\mathcal{S}$ to a neighborhood
of \ $\mathcal{S}$ in such a way that the result is
a unit gradient field (see, e.g., \cite{DMM1}--\cite{MM1} and the references therein).  We call such
extensions proper.
\begin{definition}
\label{d2.4.i} Let ${\mathcal{S}}$ be a hypersurface in $\mathbb{R}^{n}$ and ${%
\boldsymbol{\nu }}$ be the unit normal vector field on ${\mathcal{S}}$. A
vector filed $\mathcal{N}\in C^1(\Omega_\mathcal{S})$ in a neighborhood
$\Omega_\mathcal{S}$ of $\mathcal{S}$ is referred to as a
\textbf{proper extension} if $\mathcal{N}\big|_\mathcal{S}=\boldsymbol{\nu}$ and
\begin{equation}\label{e2.2.19x}
|\mathcal{N}(x)|=1,\qquad \partial _{j}\mathcal{N}_{k}(x)
     =\partial _{k}\mathcal{N}_{j}(x)\qquad \text{for all}\quad x\in
     \Omega_\mathcal{S},\quad j,k=1,\ldots,n.
\end{equation}
\end{definition}

The existence of such an extension follows from the well known existence and uniqueness result for the following boundary value problem for the \textbf{eikonal equation}: for a given hypersurface $\mathcal{S}$, find a function $\Phi_\mathcal{S}$ such that
\begin{equation}\label{e0.1}
 \begin{array}{l}
|\nabla\,\Phi_\mathcal{S}(x)|=1\qquad \forall\, x\in\Omega_\mathcal{S},\\[3mm]
\Phi_\mathcal{S}(\cx)=0\quad \text{and}\quad\nabla\,\Phi_\mathcal{S}(\cx)=\nub(\cx)
     \qquad \text{for}\quad \cx\in\mathcal{S},
 \end{array}
\end{equation}
where $\nub(\cx)$, $\cx\in\mathcal{S}$ is the unit normal vector field on
the hypersurface $\mathcal{S}$  (see, e.g., \cite[page 88-89]{MF1}). Indeed,
if $\Phi_\mathcal{S}$ is a solution to the problem \eqref{e0.1}, the gradient
$\mathcal{N}(x)=\nabla\,\Phi_\mathcal{S}(x)$, $x\in\Omega_\mathcal{S}$, is
a proper extension of the unit normal vector field $\nub(\cx)$,
$\cx\in\mathcal{S}$.

The aim of this paper is to present an elementary proof of the existence and uniqueness result for the proper extension problem and for \eqref{e0.1}, which does not rely on the theory of Hamilton-Jacobi equations, and to provide a streamlined presentations of some results discussed in \cite{DMM1}--\cite{duka}.

The paper is organized as follows. In Section 1, we recall some definitions and
introduce basic notation from the theory of hypersurfaces.
In Section 2, we present some useful properties of a proper extension of a unit
normal vector field to a hypersurface. The main result of the
paper is proved in Section 3.

\section{Hypersurfaces and their normal vectors}

\begin{definition}
\label{dA2.1.4} A subset ${\mathcal{S}}\subset \mathbb{R}^{n}$ of the
Euclidean space is called a \textbf{hypersurface} if there exist an open covering
${\mathcal{S}}=\bigcup_{j=1}^{M}{\mathcal{S}}_{j}$ and coordinate mappings
\begin{eqnarray}
\Theta _{j}\; &:&\;\omega _{j}\rightarrow {\mathcal{S}}_{j}:=\Theta
_{j}(\omega _{j})\subset \mathbb{R}^{n},\qquad \omega _{j}\subset \mathbb{R}^{n-1} \
\mbox{ are open and bounded},  \label{e1.1} \\
j &=&1,\ldots ,M,  \notag
\end{eqnarray}%
such that the corresponding differentials
\begin{equation}
D\Theta _{j}(p):=\mathrm{matr}\,[\partial _{1}\Theta _{j}(p),\ldots
,\partial _{n-1}\Theta _{j}(p)]\,,  \label{eA2.1.15}
\end{equation}%
have the full rank
\begin{equation*}
\mathrm{rank}\,D\Theta _{j}(p)=n-1,\qquad \forall p\in \omega _{j},\qquad
j=1,\ldots ,M,
\end{equation*}%
i.e., $\Theta _{j}$ are regular over the domains $\omega _{j}$ for all $j=1,\ldots ,M$.

The hypersurface is called \textbf{smooth} (or \textbf{$k$-smooth}) if the corresponding coordinate
diffeomorphisms $\Theta _{j}$ in \eqref{e1.1} are $C^{\infty }$-smooth ($k$-smooth respectively).
\end{definition}

\begin{remark}\label{r2} 
Defining the smoothness of a manifold one needs to consider transition maps like of the atlas $\Theta^{-1}_i\circ\Theta_j\;:\;\omega_i\cap\omega_j\rightarrow \omega_i\cap\omega_j$ when defining a general manifold that is not assumed a priori to be embedded into a Euclidean space. The ambient Euclidean space allows one to define hypersurfaces without re-course to transition maps. The latter can be proved to have the necessary smoothness with the help of the rank condition and the implicit function theorem.

A closed hypersurface (without boundary) in $\mathbb{R}^n$ is orientable. An elementary proof of this can be found in \cite{Sa1}.
\end{remark}

\begin{definition}
\label{dA2.1.4a} Let $k\geqslant 1$ and $\Omega \subset \mathbb{R}^{n}$ be a
bounded domain. An \textbf{implicit $C^{k}$-smooth}  hypersurface in $\mathbb{R}^{n}$ is defined as the set
\begin{equation}
{\mathcal{S}}=\left\{ t\in \Omega \;:\;\Psi _{{\mathcal{S}}}(t)=0\,\right\} ,
\label{eA2.1.14a}
\end{equation}%
where $\Psi _{{\mathcal{S}}}\,:\,\Omega \rightarrow \mathbb{R}$ is a $C^{k}$-mapping
which is regular $\nabla \,\Psi(t)\not=0$, $\forall t \in \Omega$.
\end{definition}

\begin{lemma}[{(see, e.g., \protect\cite[§\thinspace 1]{Du1})}]
\label{lA2.1.7} Definition \ref{dA2.1.4} and
Definition \ref{dA2.1.4a} of a $k$-smooth hypersurface ${\mathcal{S}}$ are
equivalent.
\end{lemma}

\begin{remark}
\label{r2.5.c} For a given hypersurface, an implicit surface function is
defined with the help of the signed distance
\begin{equation*}
\Psi _{{\mathcal{S}}}(x):=\pm \mathrm{dist}(x,{\mathcal{S}}), \qquad x\in
\Omega _{{\mathcal{S}}}\,,
\end{equation*}%
where the signs \textquotedblleft +" and \textquotedblleft --" are chosen
for $x$ \textquotedblleft above" in the direction of the unit normal vector
and \textquotedblleft below" ${\mathcal{S}}$, respectively (see \cite[§\,3]%
{DMM1}).
\end{remark}

We will need the following textbook result.
\begin{lemma}
\label{l0.1} Let ${\mathcal{S}}\subset \mathbb{R}^{n}$ be a $k$-smooth
hypersurface, $k=2,3,\ldots $, given implicitly $\Psi _{{\mathcal{S}}}({%
\scriptstyle{\mathcal{X}}})=0$ by the function $\Psi _{{\mathcal{S}}}\in
C^{k}(\Omega _{{\mathcal{S}}})$.

The $C^{k-1}$-smooth unit vector field
\begin{equation}\label{e1.59}
\boldsymbol{\nu}(\cx):=\frac{(\nabla\Psi_{\mathcal{S}})(\cx)}{
    |(\nabla\Psi_\mathcal{S})(\cx)|},\qquad {\scriptstyle{\mathcal{X}}}
    \in\mathcal{S}
\end{equation}
is normal (orthogonal) to the surface $\mathcal{S}$.
\end{lemma}

\section{Properties of a proper extension}

First note that the extension
\begin{equation}\label{ea333}
{\boldsymbol{\nu }}(x):=\frac{(\nabla \Psi _{\mathcal{S}})({x})}{|(\nabla
\Psi _{\mathcal{S}})({x})|}\, ,\qquad x\in \Omega _{\mathcal{S}}
\end{equation}
of the normal vector field ${\boldsymbol{\nu }}(x)$ (see \eqref{e1.59}) is
not in general a proper one. Indeed, let $n=2$ and let ${\mathcal{S}}$ be the ellipse
\begin{equation*}
\left\{ x=\left( x_{1},x_{2}\right) \in \mathbb{R}^{2}|\Psi _{\mathcal{S}}({x%
}_{1},x_{2}):=x_{1}^{2}+2x_{2}^{2}-1=0\right\} .
\end{equation*}

Then
\begin{eqnarray*}
\mathcal{N}(x):=\frac{(\nabla \Psi _{\mathcal{S}})({x})}{|(\nabla \Psi _{\mathcal{S}})({x})|}
=\frac{(x_{1},2x_{2})}{\sqrt{x_{1}^{2}+4x_{2}^{2}}}\, ,\\
\partial _{1}\mathcal{N}_{2}(x)=-\frac{2x_{2}\ x_{1}}{\left(x_{1}^{2}+4x_{2}^{2}\right) ^{3/2}}\, ,\\
\partial _{2}\mathcal{N}_{1}(x)=-\frac{4x_{1}x_{2}}{\left(x_{1}^{2}+4x_{2}^{2}\right) ^{3/2}}\, .
\end{eqnarray*}
Hence $\partial _{1}\mathcal{N}_{2}(x)\neq \partial _{2}\mathcal{N}_{1}(x)$
unless $x_{1}=0$ or $x_{2}=0.$
\begin{lemma}\label{l4.1}
Gunter's derivatives
\begin{equation}\label{defGunt}
\mathcal{D}_{k}:=\partial_k-\nu_k\partial_\nub
\end{equation}
satisfy the following equalities:
\begin{equation}\label{e4.2}
\mathcal{D}_{k}\nu _{j}(\cx)=\mathcal{D}_{j}\nu _{k}(\cx)\quad \text{for all}
     \quad\cx\in \mathcal{S},\quad j,\,k=1,\,2,\dots ,n.
\end{equation}
\end{lemma}
\begin{proof}
Since (see \eqref{e1.59})
\begin{equation*}
\nu_k:=\frac{\partial_k\Psi_\mathcal{S}}{|(\nabla\Psi_\mathcal{S})|}\, ,
\end{equation*}
a routine calculation gives
\begin{eqnarray*}
\mathcal{D}_j\nu_k&\hskip-3mm=&\hskip-3mm\partial_j\nu_k
    -\nu_j\partial_{\boldsymbol{\nu}}\nu_k=
    \partial_j\frac{\partial_k\Psi_\mathcal{S}}{|(\nabla\Psi_\mathcal{S})|}
    -\frac{\partial_j\Psi_\mathcal{S}}{|(\nabla\Psi_\mathcal{S})|}
    \sum_{m=1}^{n}\frac{\partial_m\Psi_\mathcal{S}}{|(\nabla\Psi_\mathcal{S})|}
    \partial_m\frac{\partial_k\Psi_\mathcal{S}}{|(\nabla\Psi_\mathcal{S})|}\\
&\hskip-3mm=&\hskip-3mm\frac{\partial_j\partial_k\Psi_\mathcal{S}}{|
    (\nabla \Psi_{\mathcal{S}})|}-\sum_{\ell=1}^{n}\frac{\partial_k\Psi_\mathcal{S}
    \partial_\ell\Psi_\mathcal{S}\partial_\ell\partial_j\Psi _{\mathcal{S}}}{|
    \nabla\Psi_\mathcal{S}|^3}-\sum_{m=1}^{n}\frac{\partial_j\Psi_\mathcal{S}
    \partial_m\Psi_\mathcal{S}\partial_m\partial_k\Psi_\mathcal{S}}{
    |(\nabla\Psi_{\mathcal{S}})|^3}\\
&&+\sum_{m=1}^{n}\sum_{l=1}^{n}\frac{\partial _{j}\Psi _{\mathcal{S}}\partial_{m}
     \Psi_{\mathcal{S}}\partial_{k}\Psi_{\mathcal{S}}\partial_{l}
     \Psi_{\mathcal{S}}\partial _{m}\partial _{l}\Psi _{\mathcal{S}}}{|(\nabla
     \Psi_\mathcal{S})|^{5}}=\mathcal{D}_{k}\nu_j
     \quad \text{for all}\quad j,\,k=1,\,2,\dots,n.
\end{eqnarray*}
The last equality holds because the expression $\mathcal{D}_j\nu_k$ turns out  to be symmetric with respect to the
indices $j$ and $k$.

One can give an alternative proof that avoids the above calculations, if one assumes the existence of a proper extension of ${\boldsymbol{\nu}}$
to a neighborhood $\Omega_\mathcal{S}$ of $\mathcal{S}$ (see the proof of \eqref{e2.5} in Lemma \ref{lemma:N} below).
\end{proof}
 %
\begin{lemma}
\label{l6.19} For a unitary (not necessarily proper) extension $\mathcal{N}(x)\in C^1(\Omega_\mathcal{S})$, $|\mathcal{N}(x)|\equiv1$
of ${\boldsymbol{\nu}}(t)$
to a neighborhood $\Omega_\mathcal{S}$ of $\mathcal{S}$, the following conditions
are equivalent:

\begin{enumerate}
\item[(i)] $\partial _{\mathcal{N}}\mathcal{N}|_{{\mathcal{S}}}=0$,

\item[(ii)] $\left[ \partial _{k}\mathcal{N}_{j}-\partial _{j}\mathcal{N}_{k}%
\right] |_{\mathcal{S}}=0\;\;$ for $\;\;k,\,j=1,\,2,\dots ,n$.
\end{enumerate}
\end{lemma}
\begin{proof}
Suppose $(i)$ holds. Then one has on $\mathcal{S}$
\begin{eqnarray*}
\partial_k \mathcal{N}_j & = & \mathcal{D}_k \mathcal{N}_j \ \ \ (\mbox{due to } (i) \mbox{ and the definition \eqref{defGunt} of } \mathcal{D}_k)  \\
& = &  \mathcal{D}_k \nu_j \ \ \ (\mbox{since } \mathcal{D}_k \mbox{ is a tangent derivative and } \mathcal{N} = \nu \mbox{ on } \mathcal{S})  \\
& = &  \mathcal{D}_j \nu_k \ \ \ (\mbox{by \eqref{e4.2}}) \\
& = &  \mathcal{D}_j \mathcal{N}_k = \partial_j \mathcal{N}_k  \ \ \ (\mbox{as above}).
\end{eqnarray*}

Suppose now $(ii)$ holds. Then  one has on $\mathcal{S}$
\begin{equation}\label{dNN}
\partial _{\mathcal{N}}\mathcal{N}_{j}=\sum_{k=1}^{n}\mathcal{N}_{k}\partial_{k}\mathcal{N}_{j}=\sum_{k=1}^{n}\mathcal{N}_{k}\partial _{j}\mathcal{N}_{k}=\frac{1}{2}\sum_{k=1}^{n}\partial _{j}\mathcal{N}_{k}^{2}=\frac{1}{2} \partial _{j}1=0 .
\end{equation}
\end{proof}

\begin{corollary}
\label{c3.3} Let $\mathcal{N}\in C^1(\Omega_\mathcal{S})$, $|\mathcal{N}(x)|\equiv1$, be a unitary (not necessarily proper) extension of
${\boldsymbol{\nu}}$ to a neighborhood $\Omega_\mathcal{S}$ of $\mathcal{S}$.

If one of conditions (i) or (ii) of Lemma \ref{l6.19} holds,  then
\begin{equation}\label{e2.3S}
\mathcal{D}_{k}\mathcal{N}_{j}(x)=\partial _{k}\mathcal{N}_{j}(x)=\partial
_{j}\mathcal{N}_{k}(x)=\mathcal{D}_{j}\mathcal{N}_{k}(x)\quad \text{for all}%
\quad x\in \mathcal{S} .
\end{equation}
\end{corollary}
\begin{proof}
The claimed equalities  follow from conditions
(i) and (ii) of Lemma \ref{l6.19} and the definition of Gunter's derivative $\mathcal{D}_{k}$.
\end{proof}

 %
\begin{lemma}\label{lemma:N}
Any proper extension $\mathcal{N}(x)$, $x\in \Omega _{{\mathcal{S}}}\subset \mathbb{R}^{n}$ of the unit
normal vector field ${\boldsymbol{\nu }}$ to the surface $\mathcal{S}\subset\Omega_\mathcal{S}$
satisfies the equality
\begin{equation}\label{e0.2}
\partial _{\mathcal{N}}\mathcal{N}(x)=0,\qquad \mbox{\rm for all}\quad x\in
\Omega _{\mathcal{S}}.
\end{equation}

Moreover, for the extensions of Gunter's derivatives $\mathcal{D}_{k}=\partial_{k} - \nu_{k}\partial _{\boldsymbol{\nu}}$ to the neighborhood $\Omega _{{\mathcal{S}}}$ of the surface ${\mathcal{S}}$
\begin{equation}\label{e2.7}
\mathcal{D}_{k}=\partial _{k}-\mathcal{N}_{k}\partial _{\mathcal{N}},\qquad
k=1,\ldots ,n,
\end{equation}
the following equalities hold
\begin{equation}\label{e2.3}
\mathcal{D}_{k}\mathcal{N}_{j}(x)=\partial _{k}\mathcal{N}_{j}(x)=\partial
_{j}\mathcal{N}_{k}(x)=\mathcal{D}_{j}\mathcal{N}_{k}(x)\quad \text{for all}\quad x\in \Omega _{\mathcal{S}}
\end{equation}and, in particular,
\begin{equation}\label{e2.5}
\mathcal{D}_{k}\nu _{j}({}{\scriptstyle{\mathcal{X}}})=\mathcal{D}_{j}\nu
_{k}({}{\scriptstyle{\mathcal{X}}})\quad \text{for all}\quad {}{\scriptstyle{\mathcal{X}}}\in \mathcal{S},\quad j,\,k=1,\,2,\dots ,n.
\end{equation}
\end{lemma}

\begin{proof} Equality \eqref{e0.2} is proved in exactly the same way as \eqref{dNN}.
Since $\mathcal{D}_{k}=\partial _{k}-\mathcal{N}_{k}\partial _{\mathcal{N}},$ \eqref{e2.3} and \eqref{e2.5} are a direct
consequence of \eqref{e0.2}.
\end{proof}

\section{\protect\bigskip Existence of a proper extension}

We prove in this section that the formula
\begin{eqnarray}\label{e2.14x}
\mathcal{N}(\cx+t\nub(\cx))=\nub(\cx), \quad x=\cx+t\nub(\cx) ,
  \qquad \cx\in\mathcal{S},\quad
     -\varepsilon<t<\varepsilon
\end{eqnarray}
defines a unique proper extension $\mathcal{N}(x)$
of the  unit normal vector field $\nub(\cx)$ from the hypersurface
${\mathcal{S}}\subset \mathbb{R}^{n}$ into a neighborhood $\Omega_\mathcal{S}$
\begin{equation}\label{OmegaS}
\Omega_\mathcal{S}:=\left\{x=\cx+t\nub(\cx)\;:\;\cx\in \mathcal{S},\quad
     -\varepsilon<t<\varepsilon\right\}
\end{equation}
of $\mathcal{S}$.

 %
\begin{theorem}\label{t4.1}
Let ${\mathcal{S}}\subset \mathbb{R}^{n}$ be a hypersurface given by an implicit
function
 \[
\mathcal{S}=\left\{\cx\in\mathbb{R}^n\;:\;\Psi_\mathcal{S}(\cx)=0\right\}.
 \]
Then the function
\begin{equation}\label{e4.3}
\Phi_\mathcal{S}(\cx+t\nub(\cx)):=t, \qquad\cx+t\nub(\cx)\in\Omega_\mathcal{S}
\end{equation}
represents a unique solution to the eikonal boundary value problem \eqref{e0.1},
while its gradient
\begin{eqnarray}\label{e2.14gr}
\nabla\Phi_\mathcal{S}(\cx+t\nub(\cx))=\nub(\cx), \quad x=\cx+t\nub(\cx) ,
  \qquad \cx\in\mathcal{S},\quad
     -\varepsilon<t<\varepsilon
\end{eqnarray}
is a unique proper extension of the  the unit
normal vector field ${\boldsymbol{\nu }}$ to the surface $\mathcal{S}\subset\Omega_\mathcal{S}$.
\end{theorem}
\begin{proof} \underline{Uniqueness.} \ Let $\mathcal{N}$ be a proper extension of $\nub$ and let $\tau \mapsto \gamma(\tau)$ be
an integral curve of $\mathcal{N}$ starting at $x \in \mathcal{S}$. Then $\gamma(0) = x$,
$\frac{d\gamma}{d\tau}(0) = \nub(x)$, and it follows from \eqref{e0.2} that
$$
\frac{d^2\gamma}{d\tau^2}(\tau) = \frac{d\mathcal{N}(\gamma(\tau))}{d\tau}  = (\partial_{\mathcal{N}}\mathcal{N})(\gamma(\tau)) = 0 .
$$
Hence $\frac{d\gamma}{d\tau}(\tau) = \mbox{const} = \frac{d\gamma}{d\tau}(0) = \nub(x)$ and
$\gamma(\tau) = \gamma(0) + \tau \nub(x) = x + \tau \nub(x)$. Therefore, $\mathcal{N}(x + \tau\nub(x)) =
\mathcal{N}(\gamma(\tau)) = \frac{d\gamma}{d\tau}(t) = \nub(x)$, i.e. \eqref{e2.14x} holds, which proves the
uniqueness of a  proper extension of ${\boldsymbol{\nu }}$. The uniqueness of a solution of \eqref{e0.1} is now
immediate. Indeed, if $\Phi_1$ and $\Phi_2$ are solutions of \eqref{e0.1}, then $\Phi_1 - \Phi_2 = 0$ on $\mathcal{S}$,
and it follows from the uniqueness of a  proper extension that $\nabla(\Phi_1 - \Phi_2) = \nabla\Phi_1 - \nabla\Phi_2 = 0$.
So, $\Phi_1 - \Phi_2 = 0$ in $\Omega_\mathcal{S}$.

\underline{Existence.} \ Our aim here is to prove  that the gradient of the function
$\Phi_\mathcal{S}$ defined by \eqref{e4.3} is an extension we need and that \eqref{e2.14gr} holds.
Let $\mathcal{S}_t$, $t \in (-\epsilon, \epsilon)$ be the $t$-level set
of $\Phi_\mathcal{S}$:
$$
\mathcal{S}_t = \{y \in \Omega_\mathcal{S} | \ \Phi_\mathcal{S}(y) = t\} = \{x + t\nub(x) | \ x \in \mathcal{S}\}.
$$
Let us show that $\nub(x)$ is normal to $\mathcal{S}_t$ at the point $x + t\nub(x)$. Using the local coordinates \eqref{e1.1} we see
that any tangential to  $\mathcal{S}_t$ vector at $\Theta_j(p) + t\nub(\Theta_j(p))$ is a linear combination of the
vectors
$$
\partial_{p_k}\left(\Theta_j(p) +  t\nub(\Theta_j(p))\right)\, ,  \ \ \ k = 1, \dots, n - 1.
$$
Taking the scalar product with $\nub(\Theta_j(p))$ we get
\begin{eqnarray*}
&& \left\langle\partial_{p_k}\left(\Theta_j(p) +  t\nub(\Theta_j(p))\right), \nub(\Theta_j(p))\right\rangle \\
&& =
\left\langle\partial_{p_k}\Theta_j(p), \nub(\Theta_j(p))\right\rangle +
t \left\langle\partial_{p_k}\nub(\Theta_j(p)), \nub(\Theta_j(p))\right\rangle \\
&& = 0 + \frac{t}2\, \partial_{p_k} |\nub(\Theta_j(p))|^2 = \frac{t}2\, \partial_{p_k} 1 = 0 , \ \ \ k = 1, \dots, n - 1.
\end{eqnarray*}
Hence $\nub(x)$ is indeed normal to $\mathcal{S}_t$ at the point $x + t\nub(x)$. Since $\mathcal{S}_t$
is defined by the equation $\Phi_\mathcal{S}(y) = t$, the gradient $(\nabla\Phi_\mathcal{S})(x + t\nub(x))$
is normal to $\mathcal{S}_t$ at the point $x + t\nub(x)$. So, there exists $\rho = \rho(x) \in \mathbb{R}$ such
that $(\nabla\Phi_\mathcal{S})(x + t\nub(x)) = \rho \nub(x)$, and it is easy to see that $\rho \ge 0$. It is left to
prove that $\rho = 1$. Since all tangential derivatives of $\Phi_\mathcal{S}$ on $\mathcal{S}_t$ are equal to 0,
we have
\begin{eqnarray*}
 |(\nabla\Phi_\mathcal{S})(x + t\nub(x))| &=& |(\partial_\nub\Phi_\mathcal{S})(x + t\nub(x))| \\
& =  & \left|\lim_{h \to 0} \frac{\Phi_\mathcal{S}(x + (t + h)\nub(x)) - \Phi_\mathcal{S}(x + t\nub(x))}{h}\right| \\
& = & \left|\lim_{h \to 0} \frac{(t + h) - t}{h}\right| = 1 .
\end{eqnarray*}
Hence $\rho = 1$, i.e. $(\nabla\Phi_\mathcal{S})(x + t\nub(x)) = \nub(x)$, $x \in \mathcal{S}$.
\end{proof}

\end{document}